\documentclass[reqno]{amsart}
\usepackage{amsmath}
\usepackage{amssymb}
 \usepackage[colorlinks=true]{hyperref}

\usepackage{color}
\usepackage{mathrsfs}
\usepackage{extarrows}
\usepackage{enumerate}

\newtheorem{theorem}{Theorem}[section]
\newtheorem{corollary}[theorem]{Corollary}

\newtheorem{lemma}[theorem]{Lemma}
\newtheorem{proposition}[theorem]{Proposition}
\newtheorem*{conjecture}{Conjecture}

\theoremstyle{definition}

\newtheorem*{remark}{Remark}

\newcommand{\ep}{\varepsilon}

\newcommand{\RR}{\mathbb{R}}

\newcommand{\NN}{\mathbb{N}}

\newcommand{\ZZ}{\mathbb{Z}}

\newcommand{\cE}{\mathcal{E}}
\newcommand{\cA}{\mathcal{A}}

\newcommand{\cU}{\mathcal{U}}

\newcommand{\cm}{\mathcal{M}}
\newcommand{\fm}{\mathfrak{M}}

\newcommand{\sE}{\mathscr{E}}

\newcommand{\sC}{\mathscr{C}}
\newcommand{\cx}{C(X)}

\newcommand{\su}{\mathscr{U}}
\newcommand{\spp}{\mathscr{P}}

\newcommand{\coo}{\cm_{\max}}

\newcommand{\sv}{\mathscr{V}}

\newcommand{\sh}{\mathscr{H}}

\newcommand{\al}{\alpha}

\newcommand{\btt}{\mathbb{T}}
\newcommand{\bttd}{\mathbb{T}^d}
\title[Equilibrium states of intermediate entropies]
      {Equilibrium states of intermediate entropies} 

\author[]{}


 \email{sunpeng@cufe.edu.cn}

\begin{document}

\maketitle\ 

\bigskip

\centerline{\scshape Peng Sun}
\medskip
{\footnotesize
 \centerline{China Economics and Management Academy}
   \centerline{Central University of Finance and Economics}
   \centerline{Beijing 100081, China}
} 

\bigskip

\begin{abstract}

We explore an approach to the conjecture of Katok on intermediate
entropies that based on uniqueness of equilibrium states,
provided the entropy function is upper semi-continuous.
As an application, we prove Katok's conjecture for Ma\~n\'e diffeomorphisms.
\end{abstract}


\bigskip

\section{Introduction}

Let $(X,d)$ be a compact metric space and
$f:X\to X$ be a continuous map. 
Denote by $\cm(f)$ 
the subspace of all invariant measures for the dynamical system
$(X,f)$ and by $\cm_e(f)$ the subset of all ergodic measures.
Denote by $C(X)$ the space of all continuous potentials on $X$,
equipped with the $C^0$ (supremum) norm $\|\cdot\|$.
For 
$\phi\in C(X)$,
denote 
by $P(\phi)=P(f,\phi)$ the topological pressure of $(X,f,\phi)$
and by 
$P_\mu(\phi)=P_\mu(f,\phi):=h_\mu(f)+\int\phi d\mu$
the pressure
of $\mu\in\cm(f)$. 
We usually omit $f$ when it is clearly fixed.
The Variational Principle states that
\begin{equation*}
P(\phi)=\sup\{P_\mu(\phi):\mu\in\cm(f)\}.
\end{equation*}
For convenience, we say that $(X,f)$ is a \emph{USC} system if the entropy function
$\mu\mapsto h_\mu(f)$ is upper semi-continuous, where $h_\mu(f)$
denotes the metric entropy of $(X,f)$ with respect to $\mu\in\cm(f)$.
Denote the set of equilibrium states for $(X,f,\phi)$ by
$$\cE(\phi):=\{\mu\in\cm(f):P_\mu(\phi)=P(\phi)\}.$$
It is well-known that when $(X,f)$ is a USC system, 
$\cE(\phi)$ is nonempty for every $\phi\in C(X)$.
Study on equilibrium states has a long history. 
Existence, uniqueness and properties of equilibrium states are
important and popular topics in dynamical systems.
For example, see \cite{Bow75, Bow08, CT, DKU, IT, Ruelle}.

As indicated in \cite[Section 9.4]{Walters}, The topological
pressure $P(f,\cdot)$ for the system
$(X,f)$, as a functional on continuous potentials,
determines 
all invariant probability measures 
and their entropies. In this article, we would like to further
explore such connections and show its application to the 
conjecture of Katok on intermediate entropies.

Denote
$$\spp(f,\phi):=\{P_\mu(f, \phi):\mu\in\cm_e(f)\}$$
and
$$\sh(f):=\spp(f,0)=\{h_\mu(f):\mu\in\cm_e(f)\}.$$
We say that $(X,f)$ has the \emph{intermediate entropy property}
if $\sh(f)\supset[0,h(f))$, i.e. for every $a\in[0,h(f))$,
there is an ergodic measure $\mu\in\cm_e(f)$ 
such that
$h_\mu(f)=a$.

\begin{conjecture}[Katok]
Let $(X,f)$ be a $C^2$ diffeomorphism on a compact Riemannian manifold.
Then 
$(X,f)$ has 
the intermediate entropy property.

\end{conjecture}

Partial results on Katok's conjecture have been obtained in
\cite{KM, Sun09, Sun10, Ures, Sun12, QS, GSW, KKK, LO, 
	Sunintent,
Sunze, Sunct}. They represent two major approaches,
which are based on hyperbolic structures
and specification-like properties, respectively. 
What we illustrate here is a new one, which shows that for USC systems,
certain results on equilibrium states 
imply the intermediate entropy property.
This approach is inspired by 
some facts in
\emph{ergodic optimization}.

Denote 
$$\su(f):=\left\{\phi\in\cx:
\text{$(X,f,\phi)$ 
has  a unique equilibrium state
}\right\}.$$
For $\phi\in\cx$, denote by $\coo(f,\phi)$ 
the set of \emph{maximizing measures} for $\phi$, i.e.
invariant
measures $\mu\in\cm(f)$ that maximizes $\int\phi d\mu$.


\begin{theorem}\label{uesintent}
Let $(X,f)$ be a USC system.
Suppose that there is $\phi\in\cx$ such that: 
\begin{enumerate}
\item $t\phi\in\su(f)$ for every $t\ge 0$.
\item $\coo(f,\phi)=\{\mu_\phi\}$ is a singleton and $h_{\mu_\phi}(f)=0$.
\end{enumerate}
Then $(X,f)$ has the intermediate entropy property.
\end{theorem}

In fact, the ergodic measures of intermediate entropies we obtain in
Theorem \ref{uesintent} are just the unique equilibrium states for
$t\phi$, whose metric entropy varies continuously with $t$. 
When $\coo(f,\phi)=\{\mu_\phi\}$ is a singleton, $\mu_\phi$ is 
the unique \emph{ground state}
and the \emph{zero temperature limit} of these equilibrium states.

We shall prove a more general result than Theorem \ref{uesintent}.
We can 
relax the uniqueness of equilibrium states and 
conclude on intermediate pressures.
For $\psi\in\cx
$, let $\sv(f,\psi)$ be the set of
all continuous potentials $\phi$ such that
$P_\mu(\psi)$ is a constant, denoted by $P^{\psi}(\phi)$, 
for all $\mu\in\sE(\phi)$.
That is,
$$\sv(f,\psi):=\left\{\phi\in\cx:
\text{
$P_\mu(\psi)=P^{\psi}(\phi)$ for every
$\mu\in\cE(\phi)$}
\right\}.$$
In particular, 
$\sv(f):=
\sv(f,0)$ is
the set of all potentials whose equilibrium states have 
equal entropies.
By definition, we have
$\psi\in\sv(f,\psi)$ 
and $\su(f)\subset\sv(f,\psi)$ for every $\psi\in\cx$.
\begin{theorem}\label{intpres}
Let $(X,f)$ be a USC system.
Suppose that there are $\psi,\phi\in\cx$ and $\al\in\RR$ such that the following holds:
\begin{enumerate}
\item $\psi+t\phi\in\sv(f,\psi)$ for every $t\ge 0$.
\item $P_\mu(\psi)\le\al$ for every $\mu\in\coo(f,\phi)$.
\end{enumerate}
Then $\spp(f,\psi)\supset[\al, P(\psi)]$.
\end{theorem}

Theorem \ref{uesintent} and \ref{intpres} apply to the 
Ma\~n\'e diffeomorphisms considered in \cite{CTmane} and
\cite{Sunct}. 
The following theorem
completely verifies
 Katok's conjecture for such systems, which improves \cite[Corollary 1.2]{Sunct}.
See Section \ref{secmane} for details.

\begin{theorem}
Ma\~n\'e diffeomorphisms 
have the intermediate entropy property.
\end{theorem}

\section{Continuity of Equilibrium States} 

Readers are also referred to \cite{Walters} for definitions and
basic properties of entropies and pressures.
The following lemma is a simple observation (cf. \cite[6.8]{Ruelle}).
\begin{lemma}
\label{prelipct}
For any system $(X,f)$ and any $\phi,\psi\in\cx$, we have
\begin{equation*}
|P(\phi)-P(\psi)|\le\|\phi-\psi\|\text{ for any }\phi,\psi\in\cx.
\end{equation*}
So $P(\cdot)$ is a (Lipschitz)
continuous function on $C(X)$.
\end{lemma}


We try to make our results as more applicable as possible.
Proposition \ref{limises} and \ref{zerotemp} are slight variations of
known results (cf. \cite[Theorem 4.2.11]{Kell} and \cite[Theorem 4.1]{Jenk}).
In particular, \eqref{munest} and \eqref{presclose} 
(hence Proposition \ref{limises} and Proposition
\ref{zerotemp}) hold 
when the measures are actually the equilibrium states
for the corresponding potentials, i.e.
$\mu_n\in\cE(\phi_n)$ or $\mu_n\in\cE(\psi+t_n\phi_n)$ for all $n$. 

\begin{proposition}\label{limises}
Let $(X,f)$ be a USC system.
Let $\{\phi_n\}_{n=1}^\infty$ be a sequence of continuous potentials
such that $\|\phi_n-\phi\|\to0$. 
Let $\{\mu_n\}_{n=1}^\infty$ be a sequence in $\cm(f)$ such that
$\mu_n\to\mu$ 
and
\begin{equation}\label{munest}
\lim_{n\to\infty}\left|P_{\mu_n}(\phi_n)-P(\phi_n)\right|=0.
\end{equation}
Then $\mu$ is an equilibrium state for $\phi$, i.e. $\mu\in\cE(\phi)$.

\end{proposition}

\begin{proof}
As $P(\cdot)$ is continuous, we have $P(\phi_n)\to P(\phi)$. So
\eqref{munest} is equivalent to the condition
\begin{equation}\label{munest2}
\lim_{n\to\infty}\left|P_{\mu_n}(\phi_n)-P(\phi)\right|=0.
\end{equation}
We fix a metric $D$ on $\cm(f)$ that induces the weak-$*$ topology.
For every $\ep>0$, as the entropy map is upper semi-continuous,
there is $\eta_\ep>0$ such that for every 
$\nu\in B(\mu,\eta_\ep)$
we have
\begin{equation}\label{localest}
h_\nu(f)<h_\mu(f)+\ep\text{ and }
\left|\int\phi d\nu-\int\phi d\mu\right|<\ep.
\end{equation}
By \eqref{munest2}, 
for 
every $\ep>0$, 
there is $N$ such that 
\begin{equation}\label{limest}
\|\phi_N-\phi\|<\ep, D(\mu_N,\mu)<\eta_\ep
\text{ and }P_{\mu_N}(\phi_N)>P(\phi)-\ep.
\end{equation}
Hence by \eqref{localest}, we have
\begin{equation}\label{intest}
h_\mu(f)>h_{\mu_N}(f)-\ep\text{ and }
\left|\int\phi d\mu_N-\int\phi d\mu\right|<\ep.
\end{equation}
By 
\eqref{limest} and 
\eqref{intest}, we have
\begin{align*}
P_\mu(\phi)&=h_\mu(f)+\int\phi d\mu
\\&> h_{\mu_N}(f)-\ep+\int\phi d\mu_N-
\left|\int\phi d\mu_N-\int\phi d\mu\right|
\\&> h_{\mu_N}(f)-\ep+\left(\int\phi_N d\mu_N-
\left|\int\phi_N d\mu_N-\int\phi d\mu_N\right|\right)-\ep
\\&>P_{\mu_N}(\phi_N)-\|\phi_N-\phi\|-2\ep
\\&>P(\phi)-4\ep.
\end{align*}
This implies that $P_\mu(\phi)\ge P(\phi)$ as it holds 
for all $\ep>0$.
Hence $\mu$ is an equilibrium state for $\phi$. 
\end{proof}

\begin{proposition}\label{zerotemp}
Let $(X,f)$ be any system and 
$\psi,\phi\in\cx$.
Let $\{t_n\}_{n=1}^\infty$ be a sequence of real numbers
such that $t_n\to\infty$.
Let  $\{\mu_n\}_{n=1}^\infty$ be a sequence in
$\cm(f)$
such that $\mu_n\to\mu$ and
\begin{equation}\label{presclose}
\lim_{n\to\infty}\frac{1}{t_n}
\left|P_{\mu_n}(\psi+t_n\phi)-P(\psi+t_n\phi)\right|=0.
\end{equation}
Then $\mu$ is a maximizing measure for $\phi$, i.e. $\mu\in\coo(\phi)$.
\end{proposition}

\begin{proof}
For every $\ep>0$, 
there is $\eta_\ep>0$ such that for every 
$\nu\in B(\mu,\eta_\ep)$
we have
\begin{equation}\label{localest2}
\left|\int\psi d\nu-\int\psi d\mu\right|<\ep
\text{ and }
\left|\int\phi d\nu-\int\phi d\mu\right|<\ep.
\end{equation}

As $t_n\to\infty$, $\mu_n\to\mu$, by \eqref{presclose},
for 
every $\ep>0$, 
there is $N$ such that for all $n>N$ we have
\begin{equation}\label{limest2}
t_n\ep>h(f),
D(\mu_n,\mu)<\eta_\ep 
\text{ and }
P(\psi+t_n\phi)-P_{\mu_n}(\psi+t_n\phi)<t_n\ep
\end{equation}
By \eqref{localest2} and \eqref{limest2},
for all $n>N$, we have
\begin{align}
\left|P_\mu(\psi)-P_{\mu_n}(\psi)\right|
\le&\left|h_\mu(f)-h_{\mu_n}(f)\right|
+\left|\int\psi d\mu-\int\psi d\mu_n\right| \notag
\\<& h(f)+\ep 
<(t_n+1)\ep\label{pmunest1}
\end{align}
and hence
\begin{align}\label{pmunest2}
&P(\psi+t_n\phi)-P_{\mu}(\psi+t_n\phi) \notag
\\\le&P(\psi+t_n\phi)-P_{\mu_n}(\psi+t_n\phi)+
\left|P_\mu(\psi+t_n\phi)-P_{\mu_n}(\psi+t_n\phi)\right| \notag
\\<& 
\left|P_\mu(\psi)-P_{\mu_n}(\psi)\right| 
+t_n\left|\int\phi d\mu-\int\phi d\mu_n\right|
+t_n\ep \notag
\\<& (3t_n+1)\ep.
\end{align}
By \eqref{pmunest1} and \eqref{pmunest2}, 
for every $\nu\in\cm(f)$ and all $n>N$, we have
\begin{align}
&t_n\left(\int\phi d\mu-\int\phi d\nu\right)\notag
\\
=
&\left(P_\mu(\psi+t_n\phi)-P_\mu(\psi)\right)
-
\left(P_\nu(\psi+t_n\phi)-P_\nu(\psi)\right) \notag
\\>&\left(P(\psi+t_n\phi)-P_\nu(\psi+t_n\phi)\right) \notag
-\left(P(\psi+t_n\phi)-P_{\mu}(\psi+t_n\phi)\right)
\\&
-\left|P_\mu(\psi)-P_\nu(\psi)\right| \notag
\\>&-(4t_n+1)\ep. \label{tnest}
\end{align}
As $t_n\to\infty$, \eqref{tnest} implies that
\begin{equation}\label{phi4ep}
\int\phi d\mu\ge\int\phi d\nu-4\ep \text{ for every }\nu\in\cm(f).
\end{equation}
This implies that $\int\phi d\mu\ge\int\phi d\nu$ as \eqref{phi4ep} holds 
for all $\ep>0$. 
Hence $\mu$ is a maximizing measure for $\phi$. 
\end{proof}

\section{Intermediate Pressures}

\begin{proposition}\label{entcts}
Let $(X,f)$ be a USC system and $\psi\in C(X)$.
Then the function $\phi\mapsto P^{\psi}(\phi)$ is 
continuous
on $\sv(f,\psi)$.
\end{proposition}

\begin{proof}

Let $\{\phi_n\}_{n=1}^\infty$ be a sequence in $\sv(f,\psi)$
such that
$\phi_n\to\tilde\phi\in\sv(f,\psi)$. 
For each $n$, take any $\mu_n\in\sE(\phi_n)$.
Let $\mu$ be the limit of a convergent
subsequence $\{\mu_{n_k}\}_{k=1}^\infty$.
As $\phi_{n_k}\to\tilde\phi$,
by Proposition \ref{limises}, we have $\mu\in\sE(\tilde\phi)$.

By Lemma \ref{prelipct},
we have
\begin{align*}
&\lim_{k\to\infty}\left|P^{\psi}(\phi_{n_k})-P^{\psi}(\tilde\phi)\right|
\\= &
\lim_{k\to\infty}\left|P_{\mu_{n_k}}(\psi)-P_{\mu}(\psi)\right|
\\\le &
\lim_{k\to\infty}\left(\left|P_{\mu_{n_k}}(\phi_{n_k})-P_{\mu}(\tilde\phi)\right|
+\left|\int(\phi_{n_k}-\psi)d_{\mu_{n_k}}-\int(\tilde\phi-\psi) d_{\mu}\right|\right)
\\\le&\lim_{k\to\infty}\left|P(\phi_{n_k})-P(\tilde\phi)\right|
+\lim_{k\to\infty}\left|\int(\phi_{n_k}-\tilde\phi) d_{\mu_{n_k}}\right|
\\&
+\lim_{k\to\infty}\left|\int(\tilde\phi-\psi) d_{\mu_{n_k}}-\int(\tilde\phi-\psi)d_{\mu}\right|
\\\le&\lim_{k\to\infty}\|\phi_{n_k}-\tilde\phi\|+\lim_{k\to\infty}\|\phi_{n_k}-\tilde\phi\|+0
\\=& 0.
\end{align*}
This implies that $P^\psi(\cdot)$ is a
continuous function on $\sv(f,\psi)$.
\end{proof}

\begin{corollary}\label{entctscor}
	Let $(X,f)$ be a USC system.
	Denote by $\mu_\phi$ the unique equilibrium state for each $\phi\in\su(f)$.
	Then the map $\phi\mapsto h_{\mu_\phi}(f)$ is 
	continuous
	on $\su(f)$.
\end{corollary}

\begin{proof}
	Note that $\su(f)\in\sv(f)$ and $P^0(\phi)=P_{\mu_\phi}(0)=h_{\mu_\phi}(f)$
	for every $\phi\in\su(f)$.
	Apply Proposition \ref{entcts} for $\psi=0$.
\end{proof}

\begin{proposition}\label{propmain}
Let $(X,f)$ be a USC system and $\psi\in\cx$. Suppose that there is a
continuous path
$\Phi:[0,T]\to\sv(f,\psi)$ and
$P_\psi(\Phi(0))<P_\psi(\Phi(T))$.
Then
\begin{equation}\label{intmed}
\spp(f,\psi)\supset[P^\psi(\Phi(0)),P^\psi(\Phi(T))]
\end{equation}


\end{proposition}

\begin{proof}
By Proposition \ref{entcts}, $P^\psi(\cdot)$ is continuous on
$\sv(f,\psi)$. Hence
$P^\psi\circ\Phi$ is continuous. Then \eqref{intmed} follows from
the Intermediate Value Theorem.
\end{proof}

Analogous to Corollary \ref{entctscor}, we have:

\begin{corollary}
Let $(X,f)$ be a USC system. Suppose that there is a
continuous path
$\Phi:[0,T]\to\su(f)$,
i.e. $\Phi(t)$ has a unique equilibrium state $\mu_t$ for each $t\in[0,T]$.
Assume that $h_{\mu_0}(f)\le h_{\mu_T}(f)$.
Then we have
\begin{equation*}
\sh(f)\supset[h_{\mu_0}(f),h_{\mu_T}(f)]
\end{equation*}
\end{corollary}

\begin{proof}[Proof of Theorem \ref{intpres}]
Take $\mu_n\in\cE(\psi+t\phi)$ for each $n\in\NN$.
As $\cm(f)$ is compact, there is a subsequence $\{t_k\}_{k=1}^\infty$
such that $t_k\to\infty$ and $\{\mu_{t_k}\}_{k=1}^\infty$ converges
to $\mu\in\cm(f)$.
By Proposition \ref{zerotemp}, $\mu$ is a maximizing measure for $\phi$
and hence $P_\mu(\psi)\le a$.
For any $b\in (a,P(\psi))$, as $\mu_{t_k}\to\mu$ and the entropy map
is upper semi-continuous,
there is $N$ such that
$P_{\mu_{t_N}}(\psi)<b$.
Note that $\Phi(t):=\psi+t\phi$ is a continuous path from $[0, t_N]$
to $\sv(f,\psi)$.
By Proposition \ref{propmain},
we have 
\begin{equation}\label{pinc}
\spp(\psi)\supset(b,P(\psi)].
\end{equation}
Then $\spp(\psi)\supset[a, P(\psi)]$
because $P_\mu(\psi)\le a$ and
\eqref{pinc} holds for any $b\in (a,P(\psi))$.
\end{proof}

\begin{corollary}\label{intentropy}
Let $(X,f)$ be a USC system with positive topological entropy $h(f)>0$.
Suppose that there is a continuous potential $\phi$ such that
the following holds:

\begin{enumerate}
\item 
$(X,f,t\phi)$ has a unique equilibrium state
for every $t\ge 0$.
\item 
$(X,f,\phi)$ has a unique maximizing measure $\mu$
with $h_{\mu}(f)=0$.
\end{enumerate}
Then $(X,f)$ has the intermediate entropy property.
\end{corollary}

\begin{proof}
Note that for $t=0$, $\mu_0$ is just the measure of maximal entropy. Hence
$h_{\mu_0}(f)=h(f)$.
Apply Theorem \ref{intpres} for $\psi=0$. 
\end{proof}

It is well known that every Axiom A system has unique equilibrium states
for H\"older potentials \cite{Bow75}. One can just pick a fixed point $p$ and take
$\phi(x):=-d(x,p)$ for every $x\in X$, which is a H\"older function
with the unique maximizing measure supported on $p$. In this case Corollary
\ref{intentropy} provides another proof that
every Axiom A system has the intermediate entropy property.

\section{Application to Ma\~n\'e Diffeomorphisms}
\label{secmane}
Following \cite{CTmane},
we consider the Ma\~n\'e family $\fm_{\rho,r}$, 
which is a class of DA (derived from Anosov)
maps first introduced by Ma\~n\'e \cite{Ma78}. They are
$C^0$ perturbations of a hyperbolic toral automorphism $f_A:\btt^d\to\btt^d$,
which are partially hyperbolic  with 1-dimensional centers.
Let $q$ be a fixed point of $f_A$.
As described in \cite{CTmane}, for each $g\in\fm_{\rho,r}$ we assume that:
\begin{enumerate}
\item 
$\rho>0$ such that the neighborhood $B(q,\rho)$ is  the support of the perturbation, i.e.
$g=f$ on $\btt^3\backslash B(q,\rho)$.
\item $r\in[0,1]$ such that if an orbit of $g$ spends a proportion at least $r$
of its time outside $B(q,\rho)$, then it contracts the vectors in the central direction.
\item The $C^0$ distance between $g$ and $f_A$ is sufficiently small,
i.e. there is a constant $\eta=\eta(f_A)$ depending only on $f_A$
such that $d_{C^0}(g,f_A)<\eta$.
In particular, this holds when $\rho$ is sufficiently small.
\end{enumerate}

For $\phi\in C(\bttd)$, $\rho, L>0$ and $r\in(0,1)$, denote
$$
\Xi(\rho,r,\phi, L):=(1-r)\sup_{B(q,\rho)}\phi+r(\sup_{\bttd}\phi+h(f_A)+L)+H(2r),
$$
where
$$
H(r):=-r\ln r-(1-r)\ln(1-r).$$
Denote
$$\sC_{\rho,r,g,L}:=\{\phi\in C(\bttd):P(g,\phi)>\Xi(\rho,r,\phi, L)\}$$
Let $\phi$ be an $\al$-H\"older potential on $\btt^d$ and 
$$|\phi|_\al:=\sup\left\{\frac{|\phi(x)-\phi(y)|}{d(x,y)^\al}:x,y\in X,
x\ne y\right\}$$
be its H\"older semi-norm. 
Denote
$$\sC^\al_M:=\{\phi\in C(\bttd):|\phi|_\al<M\}.$$


\begin{theorem}[{\cite[Theorem A and B]{CTmane}}]\label{ctm}
Let $g\in\fm_{\rho,r}$ 
and $\phi$ be an $\al$-H\"older continuous function
on $\btt^3$. 
\begin{enumerate}
\item \label{bddgen}
There 
is a constant $L=L(f_A)$ 
depending only on
$f_A$ such that 
$(\btt^d,g,\phi)$ has a unique equilibrium state as long as
$\phi\in\sC_{\rho,r,g,L}$.
\item \label{bddhol}
There is a function
$M(\rho,r)$ such that 
$M(\rho,r)\to\infty$ as $\rho,r\to 0$
and $(\bttd,g,\phi)$ has a unique equilibrium sate as long as
$\phi\in\sC^\al_{M(\rho,r)}$.
\end{enumerate}
\end{theorem}

Let $p$ be another fixed point of $f_A$ such that 
$p\notin\overline{B(q,\rho)}$. 
We can fix
$\Psi\in C(\bttd)$ such that
\begin{enumerate}
\item $\Psi(p)=0$.
\item $\Psi(x)=-1
$ for every $x\in B(q,\rho)$.
\item $\Psi(x)<0$ for every 
$x\in\bttd\backslash
\{p\}$.
\item $\Psi$ is $\al$-H\"older.
\end{enumerate}
Denote by $\mu_p$ the Dirac measure on $p$.
Then $\mu_p$ is the unique maximizing measure for $\Psi$ and
$h_{\mu_p}(g)=0$.

\begin{lemma}\label{alltuni}
There are $\rho^*>0$ and $r^*\in(0,1)$ such that
for every $g\in\fm_{\rho^*,r^*}$, we have
$t\Psi\in\su(g)$ for all $t\ge 0$.
\end{lemma}

\begin{proof}
Denote $$\beta_r:=\frac{r(h(f_A)+L)+H(2r)}{1-r}.$$
Note that $\beta_r\to 0$ as $r\to0$.
By Theorem \ref{ctm}\eqref{bddhol},
there are
$\rho^*>0$ and $r^*\in(0,1)$ such that
$$
\beta_{r^*}|\Psi|_\al<M(\rho^*,r^*)
$$
This implies that 
for every $g\in\fm_{\rho^*,r^*}$, we have
$$t\Psi\in\sC^\al_{M(\rho^*,r^*)}\subset\su(g)
\text{ for every }t\in[0,\beta_{r^*}].$$

For $t>\beta_{r^*}$, 
as $h_{\mu_p}(g)=0$,
by the Variational Principle,
we have
\begin{align*}
\Xi(\rho^*,r^*,t\Psi, L)&=(1-r)(-t)+r(h(f_A)+L)+H(2r)
\\&<0
\\&=h_{\mu_p}(g)+\int (t\Psi)d\mu_p
\\&\le P(g,t\Psi).
\end{align*}
By Theorem \ref{ctm}\eqref{bddgen},
for every $g\in\fm_{\rho^*,r^*}$, we have
$$t\Psi\in\sC_{\rho^*,r^*,g,L}\subset\su(g)\text{ for every }t>\beta_{r^*}.$$
\end{proof}

\begin{remark}
We may choose $\Psi$ such that $|\Psi|_\al$ is as small as possible
to achieve larger values of $\rho^*$ and $r^*$ as in Lemma \ref{alltuni}.
\end{remark}


\begin{theorem}
For every $g\in\fm_{\rho^*,r^*}$, the system $(\bttd,g)$
has the intermediate entropy property.
\end{theorem}

\begin{proof}
By \cite[Proposition 6]{CL}, every Ma\~n\'e diffeomorphism is
entropy expansive. So $(\bttd,g)$ is a USC system. 
As $\mu_p$ is the unique maximizing measure for $\Psi$ and
$
h_{\mu_p}(g)=0$,
the 
conclusion
follows from Lemma \ref{alltuni} and 
Corollary \ref{intentropy}.
        
\end{proof}

\section*{Acknowledgments}
The author is supported by 
National Natural Science Foundation of China
(No. 11571387) and 
CUFE Young Elite Teacher Project (No. QYP1902).
The author learned about ergodic optimization
from Yiwei Zhang and Yun Yang, and
would like to thank them
for fruitful discussions.




\end{document}